\documentclass[a4paper,11pt]{article}
\usepackage{amssymb}
\usepackage{enumerate,verbatim}
\usepackage{amsmath}
\usepackage{amsfonts}
\usepackage{amsthm}
\usepackage{epsfig}

\newtheorem{theorem}{Theorem}

\newcommand{\C}{\mathbb{C}}

\newcommand{\Z}{\mathbb{Z}}
\newcommand{\Zoladek}{{\.Zo{\l}\c{a}dek}}
\renewcommand{\div}{{\mbox{div}}}

\begin{document}
\date{}
\author{Waleed Aziz and Colin Christopher}
\title{\textbf{Local Integrability and Linearizability of Three-dimensional Lotka-Volterra Systems}}
\maketitle

\begin{abstract}
We investigate the local integrability and linearizability of three dimensional Lotka-Volterra equations at the origin.
Necessary and sufficient conditions for both integrability and linearizability are obtained for $(1,-1,1)$, $(2,-1,1)$
and $(1,-2,1)$-resonance.  To prove sufficiency, we mainly use the method of Darboux with extensions for
inverse Jacobi multipliers, and the linearizability of a node in two variables with power-series arguments in the third variable.
\end{abstract}

\section{Introduction}
In this paper we investigate the local integrability and linearizability of the origin for the
three dimensional Lotka-Volterra systems,
\begin{equation}\label{1}
\begin{aligned}
\dot x &= P = x(\lambda +ax+by+cz),\\
\dot y &= Q = y(\mu + dx +ey + fz),\\
\dot z &= R = z(\nu + gx + hy +kz),
\end{aligned}
\end{equation}
where $ \lambda, \mu, \nu \neq 0$.  Lotka-Volterra systems are widely used
with a diverse range of applications: population biology, chemical kinetics, laser physics,
plasma physics, neural networks etc. \cite{Cairo2000,CairoLlibre2000,Y.T.Christodoulides2009}

We say that the system is \textit{integrable} at the origin if there exists a change of coordinates
\[ X = x(1+O(x,y,z)),\quad Y = y(1+O(x,y,z)),\quad Z = z(1+O(x,y,z)), \]
bringing \eqref{1} to a system orbitally equivalent to its linear system:
\begin{equation}\label{1.1}
\dot X = \lambda X m,\quad \dot Y = \mu Y m,\quad \dot Z = \nu Z m,
\end{equation}
where $m = m(X,Y,Z) = 1 + O(X,Y,Z)$.  If the change of coordinates can be chosen so that $m \equiv 1$
then we say the system is \textit{linearizable}.

We are interested in the case where the eigenvalues $\lambda$, $\mu$ and $\nu$ are in
the Siegel domain and have two independent resonances.  Without loss of generality,
this means that we can assume (after a possible scaling of time) that $\lambda,\mu,\nu \in \Z$
with $gcd(\lambda,\mu,\nu) = 1$, and that $\lambda$, $\nu>0$ and $\mu< 0$.  We say that the origin has
$(\lambda:\mu:\nu)$-resonance.

Clearly, \eqref{1.1} has analytic first integrals
\begin{equation}\label{1.2}
\phi= X^{-\mu} Y^\lambda, \quad\mbox{and}\quad \psi = Y^\nu Z^{-\mu},
\end{equation}
which pull back to first integrals
\begin{equation}\label{two1stInt}
 \phi_1 = x^{-\mu}y^\lambda(1+O(x,y,z)),\quad\mbox{and}\quad \phi_2 = y^\nu z^{-\mu} (1+O(x,y,z)),
\end{equation}
of \eqref{1}.
Conversely, given two such first integrals, it is easy to construct a change of
coordinates such that $\phi$ and $\psi$ expressed in these new coordinates satisfy
\eqref{1.2}, and hence the transformed system is of the form \eqref{1.1} for some $m$.

Restricting \eqref{1} to $z = 0$ we can see the problem above as a
generalization of the problem of classifying the integrability
conditions for the system
\[ \dot x = x(\lambda +ax+by), \quad \dot y = y(\mu + dx +ey), \quad \lambda, \mu \in \mathbb{Z},\quad \lambda \mu < 0. \]

This problem was considered by several authors (Gravel and Thibault \cite{Gravel2002}, Christopher and Rousseau \cite{Christopher2004}, Li et al.\  \cite{C.Liu2004} and, for more general quadratic systems, by \.{Z}o\l\c{a}dek \cite{Zoladek97}, Fronville et al.\ \cite{AFronville1998} and Christopher et al.\ \cite{Christopher2003}) and gives a simple generalization of the Poincar\'e center-focus problem.  That is, to find conditions for which a local analytic first integral exists in a planar system.   In this case where $\lambda + \mu = 0$ we have a complexified version of the classical center-focus problem.  More recent work on integrability and linearizability of Lotka Volterra type systems with $(p,-q)$-resonance can be found in \cite{Gine2009,Gine2010,J.Gine2011,Wang2008}

Recall that in the classical center-focus problem, there are currently only two known mechanisms for integrability:
the existence of an algebraic symmetry or the existence of a Darboux integrating factor.   In the general case of $p:-q$ resonance, however,
other mechanisms appear to come into play: in particular, blow-down to a node and reduction to a Riccati equation
\cite{Christopher2003,AFronville1998}.
For the Lotka-Volterra system, many of these conditions were subsumed in \cite{Christopher2004} under a simple monodromy condition,
applied to the
neighborhoods of the invariant lines $x = 0$ and $y = 0$, and the line at infinity.  These conditions reduced essentially
to finite checks on the nature of the singularities (finite and infinite) of the system.

Our aim here is to see if a generalization of the Poincar\'e center-focus problem to higher dimension
gives a similarly simple list of integrability mechanisms.  We give a complete classification of the integrability conditions
for Lotka-Volterra equations with $(1:-1:1)$, $(2:-1:1)$ and $(1:-2:1)$ resonant critical points at the origin.  It was a surprise
to us that the problems of integrability was much harder in this case, giving rise to new forms of argument which rely less
on geometric properties than the form of the power series concerned.   This might be due to the fact that such resonant
singularities mix the saddle and node-like properties of their two-dimensional counterparts.  It would be interesting to know
whether there were more geometric ways of obtaining the sufficiency of these conditions.

Other work on 3D Lotka-Volterra equations has been done by Bobienski and \Zoladek\ \cite{Bobienski2005}, who consider the finite
singularity away from the axes planes, and give a number of mechanisms for the existence of a center in the $(i:-i:\lambda)$ case;
Cairo and Llibre \cite{Cairo2000,CairoLlibre2000}, who obtain a number of conditions for the existence of Darboux first integrals in terms of the
parameters; and Basov and Romanovski \cite{Basov2010}, who take one of eigenvalues equal to zero.  There has also been several works
devoted to systems which are homogeneous ($\lambda=\mu=\nu=0$) and hence reducible to a two-dimensional Lotka Volterra equation (\cite{Gao1998,Gonzalez2000,Ollagnier2001,Y.T.Christodoulides2009}).

We hope to extend this work to consider more general integrability phenomena in future work.  In particular, the case
where the system has a resonant integrable critical point with non-integer ratio of eigenvalues should also be approachable
using the methods here.  In these cases, the first integrals would no longer be analytic, but a product of powers of analytic functions:
$x^\alpha y^\beta z^\gamma (1+O(x,y,z))$.

The paper is organized as follows: In section 2, the Darboux method of integrability and inverse Jacobi multiplier together are explained as well as
the relation between integrability and linearizability. In Section~3, we give the complete classification of integrability and linearizability conditions of system (\ref{1}) with $(1:-1:1)$, $(2:-1:1)$ and $(1:-2:1)$resonance at the origin.


\section{Mechanisms for Integrability and Linearizability}

In this section we summarize the results we will need for understanding the integrability
mechanisms which appear in our calculations.  These methods can, in most cases, be extended to more
general situations, but we focus here on the specific context of system \eqref{1}.

\subsection{Reduction to the Poincar\'e domain}

A singular point whose eigenvalues lie in the Poincar\'e domain (that is, the convex hull of the eigenvalues does not contain the origin) can be brought to normal form via an analytic change of coordinates.  In particular, a node with two analytic separatrices can have no resonant terms in its normal form and so must be analytically linearizable.

We use this principle in two ways.  Firstly, in many cases we can choose a coordinate system so that two of the variables decouple to give a linearizable node at the origin.  If this is so, it just remains to find a linearizing transformation for the third variable via some simple power series arguments. Secondly, and more rarely, we can perform a blow down to a three-dimensional system in the Poincar\'e domain.  Since this new system is linearizable, we can find two first integrals which we can pull back to first integrals of the original system.

\subsection{Darboux Integrability and Inverse Jacobi Multipliers}

Our second main mechanism for proving integrability is the existence of a Darboux first integral.
The method is well know for two-dimensional systems \cite{C.Christopher2000,Christopher1999,Gine2010,Gravel2002},
but has been used also for
higher dimensional systems \cite{Cairo2000,CairoLlibre2000,Ollagnier1997,Y.T.Christodoulides2009}.

Let
\[X = P \frac\partial{\partial x} + Q \frac\partial{\partial y} + R \frac\partial{\partial z},\]
be the associated vector field to the system \eqref{1}.
Given a polynomial $F\in \C[x,y,z]$, a surface $F=0$ is called an invariant algebraic surface
of the system \eqref{1}, if the polynomial $F$ satisfies the equation
\begin{equation}\label{eq2}
\dot F=\mathit{X}F=P\frac{\partial F}{\partial x}+Q\frac{\partial F}{\partial y}+R\frac{\partial F}{\partial z}=C_F F
\end{equation}
for some polynomial $C_F\in$ $\mathbb{C}$. Such a polynomial is called the {\it cofactor} of the invariant algebraic curve $F=0$.
One can note that from equation \eqref{1} that any cofactor has at most degree one since the polynomial vector field has degree two.

To complete the study of integrals of parametric families, we will also need the notion of {\it exponential factor} which plays
the same role of as an invariant algebraic surface in the case when two such surfaces coalesce.
Let $E(x, y, z)=\exp(f(x, y, z)/g(x, y, z))$ where $f, g \in \C[x, y, z]$, then $E$ in an {\it exponential factor} if
\begin{equation}
\mathit{X}E=C_E E,
\end{equation}
for some polynomial $C_E$ of degree at most one.  The polynomial $C_E$ is called the {\it cofactor} of $E$.

A {\it Darboux} function is a function of the form,
\[
     D = \prod F_i^{\lambda_i} E^{\lambda_0 f/g},
\]
where the $F_i$ are invariant algebraic surfaces of the system, and $E = \exp(f/g)$ is an exponential factor.
Given a Darboux function, $D$, we can compute
\[
     X(D) = D\Big(\sum{\lambda_i C_{F_i} + \lambda_0 C_E}\Big).
\]
Clearly, the function $D$ is a non-trivial first integral of the system if and only if
the cofactors $C_{F_i}$ and $C_E$ are linearly dependent.

For Darboux integrability in two dimensions, we seek a Darboux function which is either
a first integral or integrating factor for the system.  From the latter, it is possible
to find a first integral by quadratures.

In higher dimensions, the role of the integrating factor is taken by the Jacobi Multiplier.
In the context of Darboux integrability, we usually consider the corresponding reciprocals:
inverse integrating factors, and inverse Jacobi multipliers \cite{Berrone2003}.
A function $M$ is an inverse Jacobi multiplier for the vector field $X$ if it satisfies the equation
\[
     X(M) = M \mbox{div}(X) \qquad \Longleftrightarrow \qquad \mbox{div}(X/M) = 0.
\]
A Darboux inverse Jacobi multiplier, $D$,  must satisfy $\lambda_i C_{F_i} + \lambda_0 C_E = \mbox{div}(X)$.

In three dimensions, the existence of two independent first integrals implies the
existence of an inverse Jacobi multiplier.  Conversely, given just one first integral, $\phi$,
and an inverse Jacobi multiplier, $M$, one can construct another first integral in the following
manner.

Suppose that the level surfaces $\phi=c$ are locally parameterized by some function $z = f_c(x,y)$.
Using the $x$ and $y$ coordinates to parameterize $\phi=c$, we obtain a vector field
\[P(x,y,f_c(x,y)) \frac\partial{\partial x} + Q(x,y,f_c(x,y)) \frac\partial{\partial y}.\]
It can be shown \cite{Berrone2003} that
\[ M(x,y,f_c(x,y))\frac{\partial\phi}{\partial z}(x,y,f_c(x,y))\]
is an inverse integrating factor for this vector field.  Hence, by quadratures along $\phi = c$,
we can construct a second first integral $\psi_c(x,y)$ for each value of $c$.
The function $\psi_{\phi(x,y,z)}(x,y)$ gives a second first integral of the system.

However, in our case, we would like more control on the form of this second first integral.
The following theorem allows us to give an explicit expression for the integral in certain cases.
We use the usual multi-index notation $X^I = x^i y^j z^k$ to simplify the notation.
\begin{theorem}
Suppose the analytic vector field
\begin{equation*}
x(\lambda +\sum_{|I|>0} A_{xI}X^I)\frac\partial{\partial x}
+y(\mu + \sum_{|I|>0} A_{yI}X^I)\frac\partial{\partial y}+
z(\nu + \sum_{|I|>0} A_{zI}X^I)\frac\partial{\partial z},
\end{equation*}
has a first integral
$\phi = x^\alpha y^\beta z^\gamma (1+O(x,y,z))$
with at least one of $\alpha$, $\beta$, $\gamma \neq 0$ and a Jacobi multiplier $M = x^r y^s z^t (1+O(x,y,z))$ and suppose
that the cross product of $(r-i-1,s-j-1,t-k-1)$ and $(\alpha,\beta,\gamma)$ is bounded away from
zero for any integers $i,j,k \ge 0$, then the system has a second analytic first integral
of the form $\psi = x^{1-r} y^{1-s} z^{1-t} (1+O(x,y,z))$, and hence the system \eqref{1} is integrable.
\end{theorem}

\begin{proof}
Without loss of generality, we assume that $\alpha>0$.  After an analytic change of coordinates of the form
$(x,y,z)\mapsto (x(1+O(x,y,z)),y(1+O(x,y,z)),z(1+O(x,y,z)))$, which will not alter the form of the vector field
or the inverse Jacobi multiplier, we can assume that
$\phi = X^\delta$ where $\delta = (\alpha,\beta,\gamma)$.
Furthermore, by absorbing the factor $(1+O(x,y,z))$ of $M$ into the vector field itself, we can take the
inverse Jacobi multiplier $M$ to be $X^\theta$, where $\theta=(r, s, t)$.
We take $A_I = (A_{xI},A_{yI},A_{zI})$ and write $A_{(0,0,0)} = (\lambda,\mu,\nu)$.

From the hypothesis, we can take $K>0$ such that for all $I$,
\begin{equation}\label{boundedK}
 |(\theta - I - \textbf{1})\,\times \delta| > K.
\end{equation}
Furthermore, since $\phi$ is a first integral, then $\mathit{X}\phi=0$ gives for all $I$
\begin{equation}\label{tt1}
\delta \cdot A_I=0.
\end{equation}

Since $M$ is a Jacobi multiplier, we have $X(M)=\mbox{div}(X)M$, and writing
\[\div(\mathit{X})=\frac{\partial P}{\partial x}+\frac{\partial Q}{\partial y}+\frac{\partial R}{\partial z}
=\textbf{1}\cdot \ell+\sum_{I}(I+\textbf{1})A_{xI}X^I,\]
where $\textbf{1}=(1, 1, 1)$, we see that, for all $I$,
\begin{equation}\label{tttt1}
(\theta - I - \textbf{1}) \cdot A_I=0.
\end{equation}

By hypothesis, $(\theta - I - \textbf{1})$ and $\delta$ are linearly independent, and
so \eqref{tt1} and \eqref{tttt1} imply that there exists some constants $k_I$ such that
\begin{equation}\label{tttttt1}
    A_I=k_I\,(\theta - I - \textbf{1})\,\times \delta.
\end{equation}

For ease of calculation we work with the associated 2-form $\Omega = P\,dy\wedge dz+Q\,dz\wedge dx+R\,dx\wedge dy$
rather than $X$.  In this case, a function $\phi$ is a first integral if and only if $d\phi \wedge\Omega=0$.  Now
\begin{equation}\label{t6}
\begin{aligned}
\frac{\Omega}{M}& =\frac{P\,dy\wedge dz+Q\,dz\wedge dx+R\,dx\wedge dy}{X^{\theta}}\\
& = \sum_{I} \big(k_I\,(\theta - I - \textbf{1})\,\times \delta \big)\cdot(\frac{dydz}{yz}, \frac{dzdx}{zx}, \frac{dxdy}{xy}) \: X^{I-\theta +1}\\
&  = \sum_{I} k_I\,\big((\theta - I - \textbf{1}) \cdot (\frac{dx}{x}, \frac{dy}{y}, \frac{dz}{z})\big)\wedge \big(\delta \cdot (\frac{dx}{x}, \frac{dy}{y}, \frac{dz}{z})\big) \: X^{I-\theta +1}\\
& = \sum_{I} k_I\,\big((\theta - I - \textbf{1}) \cdot (\frac{dx}{x}, \frac{dy}{y}, \frac{dz}{z}) \: X^{I-\theta +1}\big)\wedge \frac{d\phi}{\phi}\\
& =d\big( \sum_{I} k_I\: X^{I-\theta +1}\big)\wedge \frac{d\phi}{\phi}
\end{aligned}
\end{equation}
Thus, we have a formal first integral of the form
\[\psi=\sum_{I} k_I\: X^{I-\theta +1}.\]
From \eqref{boundedK} and \eqref{tttttt1}, we must have that
$|k_I|<K |A_I|$ for all $I$, and so $\psi$ is in fact analytic.
\qedhere
\end{proof}

\subsection{Integrability and Linearizability}

In some cases it is easy to deduce linearizability of a singularity from integrability.

\begin{theorem}\label{linthm}
If the system \eqref{1} is integrable and there exists a function
$\xi = x^\alpha y^\beta z^\gamma (1+O(x,y,z))$ such that $X(\xi)= k\xi$ for some constant $k=\alpha\lambda+\beta\mu+\gamma\nu$,
then the system is linearizable.
\end{theorem}

\begin{proof}
To see this, note that \eqref{1.1} implies that
$\xi_0 = \xi X^{-\alpha} Y^{-\beta} Z^{-\gamma}=1+O(x,y,z)$ satisfies $X(\xi_0^{1/k}) = \xi_0^{1/k}(1-m)$.  Hence
$(\tilde{x},\tilde{y},\tilde{z})=(X\xi_0^{\frac{\lambda}{k}},Y\xi_0^{\frac{\mu}{k}},Z\xi_0^{\frac{\nu}{k}})$ is a linearizing change of coordinates.
\end{proof}

The Lotka-Volterra equations have another property which was first noted in \cite{Christopher2004} for two dimensional systems.
However, we do not use this condition explicitly in what follows.

\begin{theorem}
Consider three dimensional Lotka-Volterra system (\ref{1}) for which the three separatrices at the origin
$x=0$, $y=0$ and $z=0$ have cofactors $L_x$, $L_y$ and $L_z$ respectively. If  $L_x$, $L_y$, $L_z$ and the
divergence $\div(X)$ are linearly independent then the origin is integrable if and only if it is linearizable.
\end{theorem}

\begin{proof}
Suppose that the origin is integrable. Then there exists two independent first integrals
$\varphi=x^{-\mu}y^{\lambda}\varphi_{1}(x, y, z)$ and $\psi=y^{\nu}z^{-\mu}\psi_{1}(x, y, z)$ where
$\varphi_{1}(x, y, z)=1+O(x, y, z)$ and $\psi_{1}(x, y, z)=1+O(x, y, z)$ are analytic.
The functions $\varphi_{1}(x, y, z)$ and $\psi_{1}(x, y, z)$ obey the equations
\[ X\varphi_{1} = \varphi_{1} L_{\varphi_{1}},  \qquad X\psi_{1} = \psi_{1} L_{\psi_{1}},\]
where
\[ L_{\varphi_{1}} = \mu L_x - \lambda L_y,  \qquad L_{\psi_{1}} = -\nu L_y + \mu L_z.\]

Since $\varphi$ and $\psi$ are first integrals, then $d\varphi \wedge \Omega=0$, $d\psi \wedge \Omega=0$ and
$d\varphi \wedge d\psi=M\,\Omega$ where $\Omega=P\,dy\wedge dz+Q\,dz\wedge dx+R\,dx\wedge dy$ as before,
and \textit{M} is a Jacobi multiplier.  One can easily show that
\[M=x^{-(\mu+1)}y^{\lambda+\nu-1}z^{-(\mu+1)}\phi(x, y, z)\]
where $\phi(x, y, z)$ satisfies $\phi(0, 0, 0)=-\mu \neq 0$ and $\phi$ has a cofactor given by the divergence
plus a linear combination of $L_x$, $L_y$ and $L_z$.
Hence the cofactors $L_x$, $L_y$, $L_z$ and $\phi$ are linearly
independent. Note that the cofactors $L_{\varphi_{1}}$,
$L_{\psi_{1}}$ and $L_{\phi_{1}}$ have no constant term. The condition on linear independence
implies that we can find a change of coordinates
$X=x\,\varphi^{\alpha_{1}}_{1} \, \psi^{\alpha_{2}}_{1} \,
\phi^{\alpha_{3}}$, $Y=y\,\varphi^{\beta_{1}}_{1} \,
\psi^{\beta_{2}}_{1} \, \phi^{\beta_{3}}$ and
$Z=z\,\varphi^{\gamma_{1}}_{1} \, \psi^{\gamma_{2}}_{1} \,
\phi^{\gamma_{3}}$ which linearizes the system.\qedhere
\end{proof}

\section{Integrability and Linearizability Conditions}

In this section, we will give a complete classification of the integrability and linearizability conditions for the origin of
\eqref{1} with $(1:-1:1)$, $(2:-1:1)$ and $(1:-2:1)$-resonance.  We note that $x=0$, $y=0$ and $z=0$ are
always invariant algebraic surfaces with cofactors $\lambda +ax+by+cz$, $\mu + dx +ey + fz$ and
$\nu + gx + hy +kz$ respectively.  The necessary conditions were found by computing the conditions for the existence of
two independent first integrals up to a given degree in the form \eqref{two1stInt} using Maple.  The degrees used $6$, $10$ and $12$ respectively.
A factorized Gr\"obner basis was then found using Reduce and finally the minAssGTZ algorithm in Singular \cite{DGPS} was used to check that the conditions found were irreducible.  To prove sufficiency, we exhibit first integrals in the form \eqref{two1stInt}.

For linearizability, we proceeded similarly: computing the conditions for the existence of a linearizing change of coordinates up to some finite order to find necessary conditions, and exhibiting a linearizing change of coordinates for sufficiency.  In this case, the first integrals can be obtained easily by pulling back the first integrals of the linearized system \eqref{two1stInt}.

\subsection{$(1:-1:1)$-resonance}

\bigskip
\begin{theorem}
Consider three dimensional Lotka-Volterra system (\ref{1}) with $(\lambda,\mu,\nu)=(1,-1,1)$. The origin is integrable if and only if one of the
following conditions are satisfied:
\begin{equation*}
\begin{aligned}
1) \: & ab - de=ac - 2ak + gk=ae + ah - de - eg=af + ak - dk - gk=\\
      & bd + bg - de - dh=bf - ch - fh + hk=bk - ce + ek - hk=\\
      & cd + cg - 2dk + fg - gk=ef - hk= 0\\
2) \: & b =  d  =  f  = h  =  0\\
3) \: & f=g=h=b-e=d-a  =  0\\
3^*) \: & b=c=d=f-k=e-h= 0\\
4) \: & b=c=d=f=k =  0\\
4^*) \: & a=d=g=h=f= 0\\
5) \: & b=e=h= 0\\
\end{aligned}
\end{equation*}
Moreover, the system is linearizable if and only if either one of the conditions (2)-(5)
or one of the following holds:
\begin{equation*}
\begin{aligned}
1.1) \: & a=c=d=f=g=k=0\\
1.2) \: & a=bk-ch=d=e-h=f-k=g=0\\
1.2^*) \: & a-d=b-e=c=dh-eg=f=k=0\\
1.3) \: & a-g=b-h=c-k=d-g=e-h=f-k=0\\
\end{aligned}
\end{equation*}
\end{theorem}

\noindent{\it Proof}. Cases 3*, 4* and 1.2* are dual to Cases 3, 4 and 1.2 under the transformation $(x,y,z)\mapsto(z,y,x)$,
and do not need to be considered separately.  The other cases are considered below.

\bigskip

\noindent\textbf{Case 1}:
If $e\neq0$, the system has an invariant algebraic plane $\ell=1+ax-ey+kz = 0$ with cofactor
$L_\ell =ax+ey+kz$, and we have two independent first integrals
\[\phi_1=xy\ell^{-1-\frac{b}{e}}, \qquad  \phi_2=yz\ell^{-1-\frac{h}{e}}.\]

When $e=0$ we have several cases.

\begin{trivlist}{}{}

\item{\textbf{i)}} $\bf b = h = 0$: This is a subcase of Case 5.

\item{\textbf{ii)}} $\bf h=0, b\neq 0$:  In this case, we have $a=f=k=g+d=0$ and we get an
exponential factor $\ell=\exp(dx-by+cz)$ with cofactor $dx+by+cz$, and first integrals
$\phi_1 = x y \ell^{-1}$ and $\phi_2 = y z$.

\item{\textbf{iii)}} $\bf b = 0, h\neq 0$:  In this case, we have $a = d = k = 0$ and $c = f(b/h-1)$.
We get an exponential factor $\ell =\exp(gx-hy+fz)$ with cofactor $gx+hy+fz$.  This gives first
integrals $\phi_1 = xy$ and $\phi_2 = yz \ell^{-1}$.

\item{\textbf{iv)}} $\bf b, h \neq 0$:  In this case, we have $a=k=0$, $c=f(b/h-1)$ and $g = d(h/b-1)$.
We get an exponential factor $\ell = \exp(dhx-bhy+bfz)$ with cofactor $dhx+bhy+bfz$.
This gives first integrals $\phi_1 = xy \ell^{-\frac{1}{h}}$ and $\phi_2 = y z \ell^{-\frac{1}{b}}$.
\end{trivlist}

\noindent\textbf{Case 1.1}:
If $e\neq0$, we have an invariant plane $1-ey = 0$, and the change of coordinates $(X, Y, Z)=(x(1-ey)^{-\frac{b}{e}}, y(1-ey)^{-1}, z(1-ey)^{-\frac{h}{e}})$ linearizes the systems. When $e=0$, we replace $(1-ey)^{-\frac{b}{e}}$ and $(1-ey)^{-\frac{h}{e}}$ above by $\exp(by)$ and $\exp(hy)$ respectively.

\medskip

\noindent\textbf{Case 1.2}:
When $e\neq 0$
The linearizing change of variables is given by
$(X, Y, Z)=(x(1-ey+fz)^{-\frac{b}{e}}, y(1-ey+fz)^{-1}, z(1-ey+fz)^{-1})$.
When $e=0$, then either $b=0$ and $f\neq 0$, and we have a linearizing change of coordinates
$(X, Y, Z)=(x(1+fz)^{c/f}, y(1+fz)^{-1}, z(1+fz)^{-1})$,
or $f=0$ and we linearize by $(X, Y, Z)=(x\exp(by-cz), y, z)$.

\medskip

\noindent\textbf{Case 1.3}:
In this case we have an invariant plane $\ell=1+ax-by+fz$ with cofactor $L_{\ell}=ax+by+fz$ and the linearizing change is $(X, Y, Z)=(x\,\ell^{-1}, y\,\ell^{-1}, z\,\ell^{-1})$.

\medskip
\noindent\textbf{Case 2}:
In this case we have an invariant algebraic plane $\ell=1-ey = 0$ with cofactor $L_\ell =ey$.
Writing $Y=y/(1-ey)$, we obtain $\dot{Y}=-Y$.  Furthermore, the equations
\begin{equation}\label{1eq16}
\dot x =  x(1+ax+cz),\qquad
\dot z =  z(1+gx+kz),
\end{equation}
gives a linearizable node. Thus, there exists $X = x(1+O(x,z))$ and
$Z = z(1+O(x,z))$ such that $\dot X =  X$ and $\dot Z =  Z$.


\bigskip
\noindent\textbf{Case 3}:
The transformation $X=x/(1+ax-by)$, $Y=y/(1+ax-by)$ gives
\[
\dot X =X(1+cz-acXz),\qquad \dot Y=Y(-1-acXz),\qquad  \dot z =z(1+kz).
\]
The first and third equations give a linearizable node, and so we can find a change of
coordinates $\tilde{X} = X(1+O(X,z))$ and $Z = z(1+O(X,z))$ such that
$\dot{\tilde{X}} =  \tilde{X}$ and $\dot Z =  Z$ (in fact $Z=z/(1+kz)$).
Since $\dot Y=Y(-1-acX(\tilde{X}, Z)\,z(Z))$,
it is sufficient to find a function $\ell(X,Y)$ such that $\dot \ell(X, Z)=X(\tilde{X},Z)z(Z)$,
then the substitution $\tilde{Y}=Ye^{ac\ell}$ will give $\dot{\tilde{Y}}=-\tilde{Y}$, and the system
is linearized.

Writing $X(\tilde{X},Z)z(Z)=\sum_{i+j>0}a_{ij}X^{i}Z^{j}$, it is easy to see that
$\ell(X, Z)=\sum_{i+j>0}\frac1{i+j}a_{ij}X^{i}Z^{j}$ gives a convergent expression for $\ell$.


\bigskip
\noindent\textbf{Case 4}:
This system has two invariant algebraic planes $1+ax=0$ and $1-ey=0$ with cofactors
$ax$ and $ey$ respectively, and a linearizing change of coordinates
\[(X, Y, Z)=(x(1+ax)^{-1}, y(1-ey)^{-1}, z(1+ax)^{-\frac{g}{a}}(1-ey)^{-\frac{h}{e}}).\]

When $a=0$, we have an exponential factor $\exp(x)$, and we replace $(1+ax)^{g/a}$ by $\exp(gx)$.
Similarly, when $e=0$, we have an exponential factor $\exp(y)$ and replace $(1-ey)^{h/e}$ by $\exp(-hy)$.

\bigskip
\noindent\textbf{Case 5}:
The equations in $x$ and $z$ are independent of $y$, and give a linearizable node
\[
\dot x =  x(1+ax+cz),\qquad
\dot z =  z(1+gx+kz).
\]
Hence there exists a change of
coordinates $X = x(1+O(x,z))$ and $Z = z(1+O(x,z))$ such that $\dot X =  X$ and $\dot Z =  Z$.
The remaining equation is given by
\[\dot y = y(-1+dx(X, Z)+fz(X, Z)).\]
Suppose there exists a function $\ell(X,Z)$ such that $\dot \ell(X, Z)=(dx(X, Z)+fz(X, Z))$,
then the transformation $Y=ye^{-\ell}$ will gives $\dot Y = -Y$.

Writing $dx(X,Z)+fz(X,Z)=\sum_{i+j>0}a_{ij}X^{i}Z^{j}$, we have
$\ell =\sum_{i+j>0}\frac{a_{ij}}{i+j}X^{i}Z^{j}$, which is clearly convergent.

\qed

\bigskip
\subsection{$(2:-1:1)$-resonance}

\bigskip
\begin{theorem}
We Consider three dimensional Lotka-Volterra system (\ref{1}) with $(\lambda,\mu,\nu)=(2,-1,1)$. Under the following (1)-(11) conditions, the origin is integrable:
\begin{equation*}
\begin{aligned}
1) \: & ef - hk =ab - ah - de + eg=ac - 3ak + 2gk=ae + ah - de - eg=\\
      & af + ak - dk - gk=bd + bg - de - 2dh + eg=bf - ch - 2fh + 2hk=\\
      & bk - ce + 2ek - 2hk=cd + cg - 3dk + 2fg - gk= 0\\
2) \: & b+e  = c = d  =  f  = k  =  0\\
3) \: & b+h=c+k=d=e-h=f-k  =  0\\
4) \: & a-d=b-e=f=g=h= 0\\
5) \: & b+e=d=f=h =  0\\
6) \: & a=d=f=g=h= 0\\
7) \: & b=c=e-h=f-k= 0\\
8) \: & b=e=h=0\\
9) \: & b=f=h=0\\
10) \: & b=c=f=k=0\\
11) \: & b=c-4k=e+h=f+k=0\\
\end{aligned}
\end{equation*}
Moreover, the origin is linearizable if and only if the system satisfied either one of the conditions (2)-(10) or one of the following conditions:
\begin{equation*}
\begin{aligned}
1.1) \: & a=c=d=f=g=k=0\\
1.2) \: & a-d=b-e=c=dh-eg=f=k=0\\
1.3) \: & a=bk-ch=d=e-h=f-k=g=0\\
1.4) \: & a-g=b-h=c-k=d-g=e-h=f-k=0\\
\end{aligned}
\end{equation*}
\end{theorem}

\noindent{\it Proof}.  We treat the cases individually.
\bigskip

\noindent\textbf{Case 1}:
If $e\neq0$, then we have an invariant surface $\ell=1+\frac{a}{2}x-ey+kz=0$
with cofactor $L_{\ell} =ax+ey+kz$.  From this we can construct
two independent first integrals
\[\phi_1=x\, y^2\,\ell^{-2-\frac{b}{e}}\quad\mbox{and}\quad \phi_2=y\, z\, \ell^{-1-\frac{h}{e}}.\]

When $e=0$ and $h\neq 0$ then we have an exponential factor $\ell =\exp((d+g)x/2-hy+fz)$ with
cofactor $(d+g)x + hy + fz$ with corresponding first integrals
\[\phi_1=x\, y^2\,\ell^{-\frac{b}{h}}\quad \mbox{and}\quad \phi_2=y\, z\, \ell^{-1}.\]

If $e = 0$ and $h = 0$ then we can assume $b\neq 0$ since otherwise we are in Case 8.
The conditions then give an exponential factor $\ell = \exp(gx+by-cz)$ with cofactor $2gx-by-cz$
and corresponding first integrals $\phi_1=x\, y^2\,\ell$ and $\phi_2=y\, z$.

\bigskip

\noindent\textbf{Case 1.1}:
If $e\neq0$, the linearizing change of coordinates is given by
\[(X, Y, Z)=(x(1-ey)^{-\frac{b}{e}}, y(1-ey)^{-1}, z(1-ey)^{-\frac{h}{e}}).\]
When $e=0$, we have an exponential factor $\exp(y)$, then replace $(1-ey)^{-\frac{b}{e}}$ and $(1-ey)^{-\frac{h}{e}}$ by $\exp(by)$ and $\exp(hy)$ respectively.

\bigskip

\noindent\textbf{Case 1.2}:
For $b\neq 0$, a linearizing change of variables is given by
\[(X, Y, Z)=(x(1+\frac{a}{2}x-by)^{-1}, y(1+\frac{a}{2}x-by)^{-1}, z(1+\frac{a}{2}x-by)^{-\frac{h}{b}}).\]
When $b=0$, either $h=0$ and $a \neq 0$, and we can linearize by
\[(X, Y, Z)=(x(1+\frac{a}{2}z)^{-1}, y(1+\frac{a}{2}z)^{-1}, z(1+\frac{a}{2}z)^{-g/a}),\]
or $a=0$ and we linearize by $(X, Y, Z)=(x, y, z\exp(by-gx))$.

\bigskip

\noindent\textbf{Case 1.3}:
This case is exactly similar to case 1.2.

\bigskip

\noindent\textbf{Case 1.4}:
In this case we have an invariant surface $\ell=1+\frac{a}{2}x-by+fz$ with cofactor $L_{\ell}=ax+by+fz$ and the linearizing change is give by $(X, Y, Z)=(x\,\ell^{-1}, y\,\ell^{-1}, z\,\ell^{-1})$.

\bigskip
\noindent\textbf{Case 2}:
In this case we have invariant planes $\ell_1=1+\frac{a}{2}x+by=0$, $\ell_2=1+\frac{a}{2}x-\frac{ab}{2}xy=0$ and $\ell_3=1+by=0$
with cofactors $L_{\ell_1} =ax-by$, $L_{\ell_2} =ax$ and $L_{\ell_3} =-by$ respectively.

The change of coordinates $(X, Y, Z)=(x\,\ell^{-1}_{1}\ell_3^{2}, y\ell^{-1}_{3}, z\,\ell^{-\frac{g}{a}}_{2}\ell^{\frac{h}{b}}_{3})$
linearizes the system.
When $b=0$ we can replace $\ell_3^{h/b}$ by $\exp(hy)$ and, when $a=0$, we replace $\ell_2^{-g/a}$ with $\exp(-x(1-by)g/2)$.

\bigskip

\noindent\textbf{Case 3}:
After the change of coordinates $(X,Y,Z)=(x,xy,xz)$, the system becomes
\begin{equation} \label{2eq8}
\dot X =  2X+aX^2-eY-fZ,\quad \dot Y =  Y(1+aX),\quad \dot Z =  Z(3+(a+g)X).
\end{equation}
The critical point at the origin of (\ref{2eq8}) is in the Poincar\'e domain and hence is linearizable via an
analytic change of coordinates which can be chosen to be of the form $(\tilde{X},\tilde{Y},\tilde{Z})=(X-eY+fZ+O(2),Y(1+O(1)),Z(1+O(1)))$.

The two first integrals $\tilde{\phi}=\tilde{X}^{-1}\,\tilde{Y}^2$ and
$\tilde{\psi}=\tilde{X}^{-2}\,\tilde{Y}\,\tilde{Z}$ of the linear system
pull back to first integrals of the form
\[\phi_1=x\,y^2(1+O(1)),\quad\mbox{and}\quad \phi_2=yz(1+O(1)).\]
The expression $\xi = x^ay^{a+g}z^{-g}$ satisfies $\dot\xi=(a-2g)\xi$ and so the system is linearizable from
Theorem~\ref{linthm}
\bigskip

\noindent\textbf{Case 4}:
The transformation $X=x/(1+ax/2-by)$, decouples the terms in $X$ and $z$ to give
\begin{equation}\label{2eq9.1}
\dot X =X\big(2+(1-aX/2)cz\big),\qquad \dot z =z(1+kz),
\end{equation}
with a linearizable node at the origin.
Thus, there is a linearizing change of coordinates
$\tilde{X} = \tilde{X}(X,z)=x(1+O(1))$, $\tilde{Z} = \tilde{Z}(z) = z(1+O(1))$
which brings \eqref{2eq9.1} to the form $\dot{\tilde{X}} = 2\tilde{X}$ and $\dot{\tilde{Z}} = \tilde{Z}$.

Taking $Y=y/(1+ax/2-by)$, we get $\dot Y=Y(-1-acXz/2)$.
It is sufficient, therefore, to find $\ell(\tilde{X},\tilde{Z})$ such that
$\dot \ell(\tilde{X}, \tilde{Z})=X(\tilde{X},\tilde{Z})z(\tilde{Z})$.
Then the substitution $\tilde{Y}=Ye^{ac\ell/2}$ will linearize the system.

If $x(X,Z)z(Z)=\sum_{i+j>0}a_{ij}X^{i}Z^{j}$, then $\ell(X, Z)=\sum_{i+j>0}\frac{a_{ij}}{2i+j}X^{i}Z^{j}$
gives a convergent expression for $\ell$.

\bigskip

\noindent\textbf{Case 5}:
A change of coordinates $(X,Y,Z)=(\frac{x}{1+by},y,\frac{z}{1+by})$,
brings the system to the form
\begin{equation}
\dot X =  X(2+aX+cZ)(1+bY),\quad \dot Y  =  -Y(1+bY),\quad \dot Z =  Z(1+gX+kZ)(1+bY).
\end{equation}
After rescaling the system by $(1+bY)$, we have $\dot Y =  -Y$
and the first and third equation give a linearizable node.
This implies the original system is integrable.  However, since $\dot Y=-Y$ the system must
also be linearizable by Theorem~\ref{linthm}.
\bigskip

\noindent\textbf{Case 6}:
A linearizing change of coordinates is given by
\[(X, Y, Z)=(x(1-ey)^{-\frac{b}{e}}(1+kz)^{-\frac{c}{k}}, y(1-ey)^{-1}, z(1+kz)^{-1}).\]
When $k=0$, we have an exponential factor $\exp(z)$, and we replace $(1+kz)^{-\frac{c}{k}}$ by $\exp(-cz)$.
Similarly, when $e=0$, we have an exponential factor $\exp(y)$ and replace $(1-ey)^{-\frac{b}{e}}$ by $\exp(by)$.

\bigskip

\noindent\textbf{Case 7}:
The system has an invariant algebraic plane $\ell=1+ax/2=0$ with cofactor $L_{\ell} =ax$ yielding a first integral
$\phi=x^{-1}\, y^{-1} \, z \, \ell^{\frac{d-g+a}{a}}$.  
We also have an inverse Jacobi multiplier
$IJM=x^{\frac{5}{2}}y^{3}\ell^{-\frac{1}{2}-\frac{2d}{a}+\frac{g}{a}}$.

When $a=0$, an exponential factor $\exp(x)$ will appear and replace $\ell^{\frac{d-g+a}{a}}$ and $\ell^{-\frac{1}{2}-\frac{2d}{a}+\frac{g}{a}}$ by $\exp(\frac{d-g}{2}x)$ and $\exp((\frac{g}{2}-d)x)$ respectively.
Theorem~1 therefore guarantees the existence of a second first integral of the
form $\psi=x^{-3/2}y^{-2}z(1+O(1))$.  From these two integrals it is easy to
construct integrals of the form required as $\phi_1=\phi^2 \psi^{-2}=xy^2(1+\ldots)$ and $\phi_2=\phi^3 \psi^{-2}=yz(1+\ldots)$.
Since $\xi=x\ell^{-1}$ satisfies $\dot\xi=2\xi$, the system is linear by Theorem~\ref{linthm}.

\bigskip

\noindent\textbf{Case 8}:
In this case, our system is
\[\dot x  =  x(2+ax+cz),\qquad \dot y =  y(-1+dx+fz),\qquad \dot z =  z(1+gx+kz).\]
The first and third equations gives a linearizable node, so it
suffices to find $\ell(X,Y)$ such that
\begin{equation}\label{2eq16}
\dot \ell(X, Z)=dx(X, Z)+fz(X, Z)
\end{equation}
The transformation $Y=ye^{-\ell}$ then linearizes the second equation.

Writing $dx(X, Z)+fz(X, Z) = \sum_{i+j>0}a_{ij}X^{i}Z^{j}$, we obtain
$\ell(X, Z)=\sum_{i+j>0}\frac{a_{ij}}{2i+j}X^{i}Z^{j}$, which is clearly convergent.

\bigskip

\noindent\textbf{Case 9}:
The system (\ref{1}) can be written as
\begin{equation} \label{2eq17}
\dot x  =  x(2+ax+cz),\qquad \dot y  =  y(-1+dx+ey),\qquad \dot z  =  z(1+gx+kz).
\end{equation}
The first and third equations in (\ref{2eq17}) give a linearizable node.
To linearize the second equation, we seek an invariant surface
of the form $\ell+\chi y = 0$ with cofactor
$dx+ey$ where $\ell=\ell(X,Z)$, $\chi=\chi(X,Z)$, and $\ell(0)=1$.
The change of variable  $Y=\frac{y}{\ell+\chi y}$ will then linearize the second equation.

To find $\ell$ and $\chi$ we therefore need to solve
\begin{equation}\label{2eq18}
\dot \chi -\chi = e\,\ell, \qquad \dot \ell = d\,x\,\ell.
\end{equation}

To find $\ell$, we write $\ell=e^{\psi}$ and solve $\dot \psi=dx$.
Let $\psi=\sum_{i+j>0}c_{ij}X^iZ^j$, then
\[\sum_{i+j>0}(2i+j)c_{ij}X^iZ^j=d x(X,Z) =d\,X+\sum_{i+j>1}d_{ij}X^iZ^j,\]
for some $d_{ij}$.  Clearly, $c_{10}=\frac{d}{2}$, $c_{01}=0$, $c_{ij}=\frac{d_{ij}}{2i+j}$ for $i+j>1$.
The convergence of $\sum_{i+j>1}d_{ij}X^iZ^j$, guarantees the convergence of $\psi$ and hence $\ell$.
Furthermore, it is clear that $\ell$ will contain no term in $Z$.

Now, writing $\ell=\sum b_{ij}X^iZ^j$, and noting that $a_{01}=0$, we find that
$\chi=\sum \frac{e}{2i+j-1}a_{ij}X^iZ^j$ gives a convergent expression for $\chi$.

\bigskip

\noindent\textbf{Case 10}:
In this case the system (\ref{1}) reduces to
\begin{equation} \label{2eq19}
\dot x  =  x(2+ax),\qquad \dot y  =  y(-1+dx+ey),\qquad \dot z  =  z(1+gx+hy).
\end{equation}
The transformation
\begin{equation}\label{2eq20}
X=\frac{x}{1+ax/2}, \qquad  Y=\frac{y}{\ell+\chi y}
\end{equation}
would linearize the first and the second equations in (\ref{2eq19}) if
$\ell+\chi y =0$ were the defining equation of an invariant algebraic surface
with cofactor $dx+ey$ where $\ell=\ell(X)$, $\chi=\chi(X)$ and $\ell(0)=1$.
This condition is equivalent to
\begin{equation}\label{2eq21.5}
\dot \chi -\chi =e\,\ell, \qquad \dot \ell = d\, x\, \ell.
\end{equation}
The second of these equations is clearly satisfied if we take $\ell = (1+ax/2)^{2/a} = (1-aX/2)^{-2/a}$ (when $a=0$ let $\ell = e^x = e^X$).
Writing $\chi=\sum a_{i}X^i$ and $\ell=\sum b_{i}X^i$,
equation (\ref{2eq21.5}) is satisfied if we set
$a_i = b_i/(2i-1)$.  The resulting function $\chi$ is clearly convergent.

To linearize the third equation, it is suffices to find $\gamma(X,Y)$ such that
$\dot \gamma =g\,x(X)+h\,y(X,Y)$, and take $Z=z\exp(-\gamma)$.

Writing,
\[
\gamma=\sum_{i+j>0} c_{ij} X^{i}Y^{j},\quad x(X)=\sum_{i>0} d_{i} X^{i}, \quad y(X,Y)=\sum_{i+j>0} e_{ij} X^{i}Y^{j},
\]
we see that we require
\[\sum (2i-j)c_{ij}X^iY^j =\sum g\,d_i X^i+\sum h\,e_{ij}X^iY^j=\sum f_{ij} X^i Y^j.\]
If $y(X,Y)$ contains no terms of the form $X^kY^{2k}$, then we can set $c_{ij}=(gd_{i}+he_{ij})/(2i-j)$
for $2i\neq j$, and $c_{ij}=0$ otherwise, to find a convergent expression for $\gamma$.

Therefore, it just remains to show that the inverse transformation $y = y(X,Y)$ from (\ref{2eq20})
contains no term like $(XY^2)^n$.
From (\ref{2eq20}),
\[y=\frac{\ell \,Y}{1-\chi Y}=\sum \ell \, \chi^n\,Y^{n+1} \]
suppose $n+1=2m$ for some $m$. It is suffices to show that $\ell \, \chi^n$ contains no term $X^m$ or in another word
$\ell \, \chi^{2m-1}$ has no term $X^m$.
Since
\[\dot \chi=\frac{d\chi}{dt}=2X\frac{d\chi}{dX},\]
then, from equation (\ref{2eq21.5}), we obtain
\begin{equation}\label{2eq21}
(\frac{2X}{2m}\frac{d\chi^{2m}}{dX}-\chi^{2m})=e\, \ell \,\chi^{2m-1}
\end{equation}
However, the term in $X^m$ in $\chi^{2m}$ clearly vanishes on the left hand side of (\ref{2eq21}), so that either
$e=0$, in which case $\chi\equiv 0$, or the coefficient of $X^m$ in $\chi^{2m-1}$ vanishes.
Thus $y$ in (\ref{2eq20}) contains no $(XY^2)^n$ and we have established the existence of a linearizing
change of coordinates.

\bigskip

\noindent\textbf{Case $\bf 11$}:
The transformation $w=yz$ brings the system to the form
\begin{equation} \label{2eq23}
\dot x  =  x(2+ax-4fz),\qquad \dot w  =  wx\tilde{d},\qquad \dot z  =  z(1+gx-fz)-ew,
\end{equation}
where $\tilde{d}=d+g$.
In this case we seek and expression $\psi=\sum w^i \psi_i(x,z)$ such that
$\dot{\psi}=x$.  If such a $\psi$ exists, then $\phi=we^{-\tilde{d}\psi}$ is a first integral.

We write the vector field as $\mathit{X}=\mathit{X_0}+\mathit{X_1}+\mathit{X_\omega}$, where
\begin{equation*}
\mathit{X_0}=x(2+ax-4fz)\frac{\partial}{\partial x}+z(1+gx-fz)\frac{\partial}{\partial z}, \quad
\mathit{X_1}=-ew\frac{\partial}{\partial z}, \quad \mathit{X_\omega}=wx\tilde{d}\frac{\partial}{\partial w}.
\end{equation*}

From $\mathit{X\psi}=x$, we get:
\begin{equation}\label{2eq24}
\mathit{X_0\psi_0}  =  x,\qquad \mathit{X_0\psi_k}+k\,x\,\tilde{d}\,\psi_k  =  -\mathit{X_1\psi_{k-1}}\quad (k>0).
\end{equation}

We now solve the equation \eqref{2eq24}.
To do this, we first show that
for every $B=\sum_{i+j>0}b_{ij}x^iz^j$, there exists an $A=\sum_{i+j>0}a_{ij}x^iz^j$,
such that $(X_0+k\tilde{d}x)\,A  = B$.  Since
\begin{equation*}
(X_0+k\tilde{d}x)A = \sum_{i+j>0}(2i+j){a_{ij}x^iz^j}+\sum_{i+j>0}(ia+jg+k\tilde{d}){a_{ij}x^{i+1}z^j}-
\sum_{i+j>0}(4i+j)f{a_{ij}x^iz^{j+1}},
\end{equation*}
we find that the $a_{ij}$ must satisfy
\begin{equation*}
(2i+j)a_{i,j}+((i-1)a+jg+k\tilde{d})a_{i-1,j}-(4i+j-1)fa_{i,j-1} = b_{i,j}.
\end{equation*}

Now in \eqref{2eq24} we can solve the equations term by term provided that the
right hand side of the equations have no constant term.  However, this follows
from the stronger observation that we can choose $\psi_i$ in \eqref{2eq24} to
be divisible by $x$.  To show this, we can suppose by induction that $x$ divides
the right hand side of \eqref{2eq24} (for $k=0$ this is immediate).  This implies that
$z(1+gx-fz)\frac{\partial\psi_k}{\partial z}$
is divisible by $x$ and this in turn shows that $\frac{\partial\psi_k}{\partial z}$ is divisible by $x$.
Writing $\psi_k(x, z)=g(z)+x\,h(x, z)$, we see that $g'(z) = 0$, so that $g$ is a constant.
Clearly $\psi_k-g$ also satisfies \eqref{2eq24} and we proceed by induction.

There is therefore no obstruction to solving these equations \eqref{2eq24} recursively,
and standard majorization techniques imply that the resulting series will be convergent.

Thus, we find a first integral $\phi_1=yze^{-\tilde{d}\psi}$ of the original system.
The system also has an inverse Jacobi multiplier
\[x^{\frac{3}{2}}y^{1+\frac{g+a/2}{\tilde{d}}}z^{\frac{g+a/2}{\tilde{d}}}.\]
Theorem~1 therefore guarantees a second first integral
\[\psi=x^{-\frac{1}{2}}y^{-\frac{g+a/2}{\tilde{d}}}z^{1-\frac{g+a/2}{\tilde{d}}}\big(1+O(1)\big),\]
from which we deduce the following first integral
\[\qquad \phi_2=\phi_1^{\frac{2d-a}{\tilde{d}}}\psi^{-2} = xy^2(1+O(1)).\]

When $\tilde{d}=0$ then $x^{3/2}y$ is an IJM and we proceed as before.
 
\qed

\bigskip

\subsection{$(1:-2:1)$-resonance}

\bigskip
\begin{theorem}
The origin of the three dimensional Lotka-Volterra system (\ref{1}) with $(\lambda,\mu,\nu)=(1,-2,1)$ is integrable if and only if on of the following conditions
holds:
\begin{equation*}
\begin{aligned}
1) \: & ab + ah - de - eg=ac - 2ak + gk=\\
             & ae + 2ah - de - 2eg=af + 2ak - dk - 2gk=\\
             & bd + 2bg - de - dh -eg=bf+ce-2ch-ek-fh+2hk=\\
             & bk - ce + ek - hk=cd + 2cg - 2dk + fg - 2gk=\\
             & ef + ek - 2hk=0\\
2) \: & a=d=f + k=g=h=0\\
3) \: & a - 2g=b + e=c - 3k=d=f + k=h=0\\
4) \: & c=d=f=g=0\\
5) \: & a=d=f=g=0\\
5^*) \: & c=d=f=k =0\\
6) \: & b - h=d=f=0\\
7) \: & a - g=c - k=d=f= 0\\
8) \: & b + h=c + k=d=e - h=f - k=g= 0\\
9) \: & b + 3h=c - 3k=d=e + 2h=f + 2k=g= 0\\
10) \: & b=c=d=e - h=f - k= 0\\
10^*) \: & a - d=b - e=f=g=h= 0\\
11) \: & b=e=h= 0\\
12) \: & a + d=b=cd - 2dk - gk=f + k=h= 0\\
13) \: & a - d=b - e=f + k=g=h= 0\\
13^*) \: & a + d=b=c=e - h=f - k=0\\
14) \: & c - 2k=d=f + k=g=h=0\\
14^*) \: & 2a - g=b=c=2d + g=f=0\\
15) \: & a + g=b + 3h=c - 3k=d=e + 2h=f + 2k= 0\\
16) \: & 3a - g=b=2c - k=3d + g=e + h=f= 0\\
17) \: & a + g=b + h=c=d + g=e + h=f= 0\\
18) \: & 3a - g=3b + h=c=3d + 2g=3e - 2h=f= 0\\
19) \: & 3a - g=3b + h=c + k=3d + 2g=3e - 2h=f=0\\
20) \: & b=d=f+k=g=h=0\\
20^*) \: & a + d=b=c=f=h= 0\\
21) \: & a + d = b = c = f = k = 0\\
\end{aligned}
\end{equation*}
Moreover, the system is linearizable if and only if either one of the conditions (2)-(8), (10)-(14), (16)-(17), (20)-(21) or one of the following holds:
\begin{equation*}
\begin{aligned}
1.1) \: & a=bk-ch=d=e-h=f-k=g=0\\
1.1^*) \: & a-d=b-e=c=dh-eg=f=k=0\\
1.2) \: & a-g=b-h=c-k=d-g=e-h=f-k=0\\
\end{aligned}
\end{equation*}
\end{theorem}

\noindent{\it Proof}. Cases 1*, 5*, 10*, 13*, 14* and 20.1* are dual to Cases 1, 5, 10, 13, 14 and 20.1 under the transformation $(x,z)\mapsto (z,x)$, and we do not consider them further.  The other cases are considered below.

\bigskip

\noindent\textbf{Case $\bf 1$}:
If $e\neq 0$, the system has an invariant algebraic plane $\ell=1+ax-\frac{e}{2}y+kz=0$
with cofactor $L_\ell=ax+ey+kz$ and produces two independent first integrals
$\phi_1=x^2 \, y\, \ell^{-1-\frac{2b}{e}}$ and $\phi_1=y\,z^2 \, \ell^{-1-\frac{2h}{e}}$.

If $e=0$, we have several sub cases:

\noindent \textbf{i) $\bf k\neq 0$}: We have $b=h=0$ and the system has first integrals
$\phi_1=x^2 \, y\, \ell^{-(\frac{f+2c}{k})}$ and $\phi_2=y\,z^2 \, \ell^{-2-\frac{f}{k}}$.

\noindent \textbf{ii) $\bf h\neq 0$}: We have $a=k=0$ and we have an exponential factor $\ell=\exp(-(d+2g)x+hy-fz)$ with cofactor
$L_{\ell}=-(d+2g)x-2hy-fz$ which yields first integrals
$\phi_1=x^2 \, y\, \ell^{\frac{b}{h}}$ and $\phi_2=y\,z^2 \, \ell$.

\noindent \textbf{iii) $\bf h=k=0$}: If $b=0$ we are in Case 11, so we assume that $b\neq 0$ which implies that $a=d+2g=f=0$.  Then there exists an exponential factor $\ell=\exp(2gx+by-2cz)$ with cofactor $L_{\ell}=2gx-2by-2cz$ which yields two first integrals
$\phi_1=x^2 y\, \ell$ and $\phi_2=y z^2$.

\bigskip

\noindent\textbf{Case $\bf 1.1$}:
If $e\neq0$, the change of coordinates $(X, Y, Z)=(x(1-\frac{e}{2}y+fz)^{-\frac{b}{e}}, y(1-\frac{e}{2}y+fz)^{-1}, z(1-\frac{e}{2}y+fz)^{-1})$ linearizes the system. When $e=0$, then either $k=0$ or $b=0$.  In the first case, we can linearize via $(X,Y.Z)=(x\,\exp(\frac{1}{2}by-cz),y,z)$, and in the second, taking $k\neq 0$, we can linearize via $(X,Y,Z)=(x(1+kz)^{-\frac{c}{k}},y(1+kz)^{-1},z(1+kz)^{-1})$.

\bigskip

\noindent\textbf{Case $\bf 1.2$}:
In this case we have an invariant plane $\ell=1+ax-\frac{b}{2}y+cz$ with cofactor $L_{\ell}=ax+by+cz$ and the linearizing change is $(X, Y, Z)=(x\,\ell^{-1}, y\,\ell^{-1}, z\,\ell^{-1})$.

\bigskip

\noindent\textbf{Case $\bf 2$}:
The system
has an invariant algebraic plane $\ell =1-\frac{e}{2}y-fz = 0$ with  cofactor
$L_{\ell} =ey-fz$.  The substitution
\begin{equation}\label{3eq3}
Y=y\,\ell^{-1}\,(1-fz)^2, \qquad  Z=\frac{z}{1-fz}.
\end{equation}
gives $\dot Y = -2Y$ and $\dot Z = Z$, and the changes of coordinates $X=xe^{-\phi}$
will give $\dot X = X$ if and only if
\begin{equation}\label{eq4}
\dot \phi(Y, Z) = -2Y \frac{\partial \phi}{\partial Y}+Z \frac{\partial \phi}{\partial Z}=b\, y(Y, Z)+c\, z(Z).
\end{equation}

An obstruction to the existence of $\phi$ is possible only if $y=y(Y,Z)$ in (\ref{eq4})
contains a term of the form $(YZ^2)^n$. According to (\ref{3eq3}),
one can find
\[y=\frac{Y(1+fZ)}{1+\frac{e}{2}Y(1+fZ)^2}=\sum_{n\ge1}(-\frac{e}{2})^{n-1} Y^{n} (1+fZ)^{2n-1},\]
and hence $y$ contains no term of the form $(YZ^2)^n$.

\bigskip

\noindent\textbf{Case $\bf 3$}:
If $g=0$ then we obtain a sub-case of Case 2 and if $f=0$ we obtain a subcase of Case 5*.
Hence, we shall assume that $f$ and $g$ are both non-zero.
The system has two first integrals:
\[\phi=\frac{xy(gexy-2gx+2fz-2f^2z^2)}{(1-2gexy+2gx-2fz+f^2z^2)^2}\]
and
\[\psi=\frac{x^2y^2z^2}{(1-2gexy+2gx-2fz+f^2z^2)^3}.\]
From these, we obtain two first integrals of the desirable form
\[\phi_1=-\frac{\phi}{2g}+\frac{f}{g}\sqrt{\psi}=x^2\,y(1+\ldots)\]
and
\[\phi_2=\frac{\psi}{\phi_1}=y\,z^2(1+\ldots).\]

The expression $\xi = xyz^{-2}$ satisfies $\dot\xi=-3\xi$ and so the system is linearizable by Theorem~\ref{linthm}.
\bigskip

\noindent\textbf{Case $\bf 4$}:
The system (\ref{1}) reduces to
\begin{equation} \label{3eq7}
\dot x = x(1+ax+by),\quad \dot y = y(-2+ey),\quad \dot z = z(1+hy+kz).
\end{equation}
The change of coordinates $Y=y/(1-ey/2)$ gives $\dot Y = -2Y$. To linearize the first equation,
we seek an invariant algebraic surface $A(Y)+B(Y)x=0$ with cofactor $ax+by$, where $A$ and $B$ are analytic with $A(0)=1$.
Thus we seek $A$ and $B$ to satisfy the equation
\begin{equation}\label{3eq8}
-2Y\,A'(Y)=b\,y\,A(Y),\qquad -2Y\,B'(Y)+B(Y)=a\,A(Y).
\end{equation}
The first equation gives $A = (1-ey/2)^{b/e}$ (or $\exp(by/2)$ when $e=0$).

Writing $A=\sum_{i\ge0} a_{i}Y^{i}$, we have $B=\sum_{i\ge0} a\frac{a_i}{1-2i}Y^{i}$, which is clearly convergent.

The substitution $X=x/(A+Bx)$ linearizes the first equation of (\ref{3eq7}).  In the same way, we
can find an invariant surface $C(Y)+D(Y)z=0$ with cofactor $hy+kz$ to so that $Z=z/(C+Dz)$ linearizes
the third equation of (\ref{3eq7}).

\bigskip

\noindent\textbf{Case $\bf 5$}:
In this case the system (\ref{1}) becomes
\begin{equation*}
\dot x = x(1+by+cz),\quad \dot y = y(-2+ey),\quad \dot z = z(1+hy+kz),
\end{equation*}
The substitution $Y=y/(1-ey/2)$ linearizes the second equation.
As in the previous case, we can find an invariant surface $C(Y)+D(Y)z=0$ with cofactor $hy+kz$ where
$C$ and $D$ are analytic with $C(0)=1$. Then
$Z = z/(C+Dz)$ linearizes the third equation.  Such a $C$ and $D$ must satisfy
\begin{equation}\label{eqqq10}
-2Y\,C'(Y)=h\,y\,C(Y),\qquad -2Y\,D'(Y)+D(Y)=k\,C(Y).
\end{equation}

We seek a transformation $X=xe^{-\phi}$ which linearizes the first equation.
Such a $\phi$ satisfies
\begin{equation}\label{eq11}
\dot \phi(Y, Z)=by(Y)+cz(Y, Z).
\end{equation}
Taking $\phi = \sum a_{ij}Y^iZ^j$, we find that $\dot\phi = \sum (2i-j)a_{ij}$.
It is clear that \eqref{eq11} can be solved analytically for $\phi$ as long as
$z(Y, Z)$ contains no term of the form $(YZ^2)^n$.  Since
\[
z(Y,Z)= \frac{ZC(Y)}{1-ZD(Y)} = \sum_{i\ge0} CD^{i-1}Z^i,
\]
it is sufficient to show that $CD^{2n-1}$ contains no term in $Y^n$.

However, from the defining equation \eqref{eqqq10}, we have
\[
   CD^{2n-1} = \frac{1}{k}\Big( -2Y\frac{\partial D}{\partial Y} + D \Big)D^{2n-1} =
   \frac{1}{k}\Big( -\frac{Y}{n}\frac{\partial D^{2n}}{\partial Y} + D^{2n} \Big),
\]
which clearly does not contain any term of degree $Y^n$.   When $k=0$, we have $D = 0$ from
\eqref{eqqq10} and the corresponding result is trivial.

\bigskip

\noindent\textbf{Case $\bf 6$}:
In this case, the change of variables
$X=xm(y)$ and $Z=zm(y)$ for some analytic function $m(y)$ with $m(0)=1$, brings the system to the form
\begin{equation} \label{3eq13}
\begin{aligned}
\dot X & =  X(1+ax+cz+by+\frac{m'}{m}y(-2+ey)),\\
\dot y & =  y(-2+ey),\\
\dot Z & =  Z(1+gx+kz+by+\frac{m'}{m}y(-2+ey)).
\end{aligned}
\end{equation}
We choose $m$ so that
\[1+by+\frac{m'}{m}y(-2+ey)=\frac{1}{m}.\]
This equation has an explicit solution,
\[m=-\frac12\sqrt{y}(1-ey/2)^{-(\frac{b}{e}+\frac{1}{2})}\,\int y^{-\frac{3}{2}}(1-ey/2)^{\frac{b}{e}-\frac{1}{2}},\]
which can be seen to be analytic in $y$ with $m(0)=1$, taking the definite integral as a Laurent series in odd powers of $\sqrt{y}$.
After scaling by $m(y)$, the system becomes,
\begin{equation*}
\dot X  =  X(1+aX+cZ),\qquad \dot y  =  y(-2+ey)m(y),\qquad \dot Z  =  Z(1+gX+kZ).
\end{equation*}
Clearly the first and third equations gives a linearizable node.  The second equation
can be linearized by a substitution $Y = \ell(y)$, such that
\[y(2-ey)m(y)\frac{d\ell(y)}{dy} = 2\ell(y), \quad \ell(0)=0, \quad \ell'(0)=1.\]   This is clearly
solvable analytically once we have determined $m(y)$ as above.

\bigskip

\noindent\textbf{Case $\bf 7$}: 
In this case the system has an invariant algebraic plane $\ell=1-ey/2=0$ with cofactor $L_\ell=ey$ producing a first integral
\[\phi=x\,z^{-1}\,\ell^{\frac{h-b}{e}}\]
It is easy to see that we have an inverse Jacobi multiplier
\begin{equation}
IJM=y^{\frac{3}{2}}z^{3}(1-ey/2)^{\frac{b}{e}-\frac{2h}{e}+\frac{1}{2}}
\end{equation}
When $e=0$, we take $\phi=x\,z^{-1}\,\exp(\frac{b-h}{2}y)$ and $IJM=y^{\frac{3}{2}}\,z^{3}\,\exp(\frac{2h-b}{2}y)$ above.
By Theorem~1, we can obtain a second first integral $\psi=xy^{-\frac{1}{2}}z^{-2}\big(1+O(1)\big)$. Now we can construct two independent first integrals of the desired form as
\[\phi_1=\phi^{4} \psi^{-2}=x^2y(1+m(x, y, z))\quad\mbox{and}\quad\phi_2=\phi^{2} \psi^{-2}=yz^2(1+\hat{m}(x, y, z)).\]
In this case the linearizing change is given by
\[(X, Y, Z)=(x\,(1+m)^{\frac{1}{2}}(-2+ey)^{\frac{1}{2}}, y\,(-2+ey)^{-1}, z\,(1+\hat{m})^{\frac{1}{2}}(-2+ey)^{\frac{1}{2}}).\]

\bigskip

\noindent\textbf{Case $\bf 8$}:
This case has invariant algebraic surfaces $\ell_1=1+\frac{b}{2}y-cz = 0$ and $\ell_2=1+ax+abxy-\frac{ac}{2}xz = 0$,
with cofactors $L_{\ell_1} =-by-cz$ and $L_{\ell_2} =ax$,
which allow us to find two independent first integrals:
\[\phi_{1}=x^{2}\, y\,\ell_{1}\ell^{-2}_{2}\qquad \mbox{and}\qquad
\phi_{2}=y\, z^{2}\,  \ell^{-3}_{1},\]
and linearizing change of coordinates $(X, Y, Z)=(x\,\ell_1\ell^{-1}_{2}, y\,\ell^{-1}_{1}, z\,\ell^{-1}_{1})$.

\bigskip

\noindent\textbf{Case $\bf 9$}:
In this case $\ell=ax(1+\frac{k}{2}z)-ahxy+(1+hy+kz)^2 = 0$ is an invariant algebraic surface
with  cofactor $L_\ell =ax-4hy+2kz$, giving rise to the first integrals
\[\phi_1=x^{2}\,y\,\ell^{-2}\qquad\mbox{and}\qquad \phi_2=y\,z^{2}.\]

\bigskip

\noindent\textbf{Case $\bf 10$}:
The change of variables $(Y,Z)=(\frac{y}{1+fz-ey/2},\frac{z}{1+fz-ey/2})$ gives a new system
\begin{equation*}
\dot x = x(1+ax),\quad \dot Y = -2Y(1+fgxZ),\quad \dot Z = Z(1+gx-fgxZ),
\end{equation*}
The first and the third equations obviously gives a node and therefore there is a
linearizing change of coordinates $\hat{X} = \hat{X}(x)$, $\hat{Z}=\hat{Z}(x,Z)$.
To linearize the second equation, it is suffices to find $\psi(\hat{X},\hat{Z})$
such that $\dot \psi=fgxZ$ and use the transformation $\hat{Y}=Ye^{2\psi}$.
Setting $\psi(\hat{X}, \hat{Z})=\sum_{i+j>0}a_{ij}\hat{X}^{i}\hat{Z}^{j}$ and
$x(\hat{X}, \hat{Z})Z(\hat{X}, \hat{Z})=\sum b_{ij}\hat{X}^{i}\hat{Z}^{j}$,
we find that $a_{ij}=d_{ij}/(i+j)$, giving a convergent expression for $\psi$.

\bigskip

\noindent\textbf{Case $\bf 11$}:
Then the system reduces to
\begin{equation*}
\dot x = x(1+ax+cz),\quad \dot y = y(-2+dx+fz),\quad \dot z = z(1+gx+kz),
\end{equation*}
The first and third equations gives a linearizable node.
We denote the linearizing coordinates by $X$ and $Z$.

The transformation $Y=ye^{-\phi}$ will linearize the second equation
if $\phi$ can be chosen so that $\dot \phi(X, Z)=dx(X, Z)+fz(X, Z)$.
Let $\phi(X, Z)=\sum_{i+j>0}a_{ij}X^{i}Z^{j}$ and $dx(X, Z)+fz(X, Z) = \sum_{i+j>0}b_{ij}X^{i}Z^{j}$,
then we require $a_{ij}=b_{ij}/(i+j)$, which gives a convergent expression for $\phi$.

\bigskip

\noindent\textbf{Case $\bf 12$}:
The system (\ref{1}) reduces to
\begin{equation*}
\dot x = x(1-dx+cz),\quad \dot y = y(-2+dx+ey-kz),\quad \dot z = z(1+gx+kz).
\end{equation*}
The system has invariant algebraic surfaces
$\ell_1=1-dx-\frac{e}{2}y+kz=0$ and $\ell_2=1-\frac{e}{2}y-\frac{ed}{2}xy+\frac{ke}{2}yz=0$
with cofactors
$L_{\ell_1} =-dx+ey+kz$ and $L_{\ell_2} =ey$.
The substitution $Y=y\,\ell_1\,\ell_{2}^{-2}$
linearizes the second equation, and the first and third equations define
a linearizable node.
\bigskip

\noindent\textbf{Case $\bf 13$}:
In this case system (\ref{1}) becomes
\begin{equation*}
\dot x = x(1+ax+by+cz),\quad \dot y = y(-2+ax+by-kz),\quad \dot z = z(1+kz),
\end{equation*}
The system has an invariant algebraic plane $\ell=1+kz=0$ with cofactor $L_\ell =kz$.
By mean of a change of variables
\[(X, Y, Z)=(x(1+kz)^{-\frac{c}{k}}, y(1+kz), z(1+kz)^{-1})\]
we arrive to the system
\begin{equation}\label{001}
\begin{aligned}
\dot X = & X(1+aX(1-kZ)^{-\frac{c}{k}}+bY(1-kZ)),\\ \dot Y = & Y(-2+aX(1-kZ)^{-\frac{c}{k}}+bY(1-kZ)),\\ \dot Z = & Z,
\end{aligned}
\end{equation}
with first integral $\phi=X^{-1}YZ^3$ and inverse Jacobi multiplier $IJM=X^{2}Y$.

We cannot apply Theorem~1, as there are non-negative integer values of $i$, $j$ and $k$ for which the cross product of $(1-i,-j,-1-k)$ and $(-1,1,3)$ is zero.  However, this is only true when, $(1-i,-j,-1-k)=\alpha\,(-1, 1, 3)$ for some $\alpha$.
Clearly, the only possibility is when $\alpha=-1$, $i=0$, $j=1$ and $k=2$.

However, in this case, the proof of Theorem~1 will still work as long as $A_{(0,1,2)}=0$ because, in this case, \eqref{tttttt1} will still hold.  But it is clear that in \eqref{001} that there are no terms in $YZ^2$ in the cofactors of $X$, $Y$ or $Z$, so indeed $A_{(0,1,2)}=0$ and we have a second first integral of the form $\psi = X^{-1}Z(1+O(X,Y,Z))$.  We get first integrals in the required form by pulling back
the first integrals $\phi_1=\phi/\psi^3=X^2Y(1+O(X,Y,Z))$ and $\phi_2=\phi/\psi=YZ^2(1+O(X,Y,Z))$ to the original coordinates.
The substitution $(\hat{X},\hat{Y},\hat{Z})= (Z/\psi,\phi_2/Z^2 ,Z)$ linearizes the system.

When $k=0$, we replace $(1+kz)^{-\frac{c}{k}}$ by $\exp(-cz)$ and proceed similarly.

\bigskip

\noindent\textbf{Case $\bf 14$}:
In this case we have two invariant algebraic planes $\ell_1=1-\frac{e}{2}y-fz=0$ and $\ell_{2}=1-fz=0$
with cofactors $L_{\ell_{1}} =ey-fz$ and $L_{\ell_{2}} =-fz$.
This allow us to construct a first integral $\phi=y\,z^{2}\ell_{1}^{-1}$
and inverse Jacobi multiplier $IJM=x^{2}\,z^{-2}\,\ell_{1}^{2-\frac{b}{e}}\:\ell_{2}^{\frac{b}{e}-1}$.
Now, Theorem~1 guarantees a second first integral $\psi=x^{-1}yz^3 (1+O(x,y,z))$.  We write the first
integrals in the desired form as
\[\phi_1=\phi^{3}\psi^{-2}=x^2\,y\big(1+m(x,y,z)\big)\qquad\mbox{and}\qquad \phi_2=\phi=y\,z^2(1+\hat{m}\big(x,y,z)\big).\]
The following change of coordinates linearizes the system
\[(X,Y,Z)=(x\,(1+m)^{\frac{1}{2}}(1-fz)^{-1}(1+\hat{m})^{-\frac{1}{2}},y\,(1-fz)^{2}(1+\hat{m}),z\,(1-fz)^{-1}).\]

\bigskip

\noindent\textbf{Case $\bf 15$}:
In this case  $\ell_1=(1+hy+kz)^2-kgxz=0$ and $\ell_2=gx(-2+gx+2hy-2kz)+(1+hy+kz)^2=0$
are invariant algebraic surfaces with cofactors $L_{\ell_1} =-4hy+2kz$ and
$L_{\ell_2} =-2gx-4hy+2kz$.  The two first integrals are given by
\[\phi_{1}=x^2\,y\,\ell^{-1}_{1}\,\ell^{-1}_{2}\qquad\mbox{and}\qquad \phi_{2}=y\,z^2\,\ell^{-1}_{1}\,\ell_{2}.\]

\bigskip

\noindent\textbf{Case $\bf 16$}:
The system has invariant algebraic surfaces $\ell_1 =1+2ax+2cz+a^2x^2-2ceyz=0$ and
$\ell_2 =2a^2x^2+2cz-ceyz+2ax=0$
with cofactors $L_{\ell_1}=2ax+2cz$ and $L_{\ell_2}=1+2ax+2cz$.
This gives a first integrals $\phi=x\,y\,z\ell^{{-\frac{3}{2}}}_{1}$ and $\psi = x^{-1}\ell_2 \ell_1^{-1/2}$, though the latter integral is
not of the required form.  However, when $z=0$, we have $\psi=2a$.  Hence, $\xi=(\psi-2a)x/(2cz) = 1+O(1)$ is analytic and satisfies $\dot\xi = \xi(-3ax+ey-cz)$.  From this we construct a linearizing change of coordinates $(X,Y,Z)=(x\,\ell_1^{-1/2},y\,\ell_1^{-1/2}\xi^{-1},z\,\ell_{1}^{-1/2}\xi)$.

When $c=0$ we have invariant surfaces $\ell_3=1+ax$ and $\ell_4=1+ax-ey/2$ with cofactors $ax$ and $ax+ey$, giving a linearizing change of coordinates $(X,Y,Z)=(x\,\ell_3^{-1},y\,\ell_3^2\ell_4^{-1},z\,\ell_3^{-4}\ell_4)$.

\bigskip

\noindent\textbf{Case $\bf 17$}:
This case has invariant algebraic surfaces $\ell_1 =1+ax-\frac{b}{2}y=0$ and $\ell_2 =1+kz+\frac{ak}{2}xz-bkyz=0$,
with cofactors $L_{\ell_1}=ax+by$ and $L_{\ell_2}=kz$.  From these we obtain first integrals
$\phi_1=x^2 \, y\, \ell^{-3}_{1}$ and $\phi_2=y\,z^2 \,  \ell^{-1}_{2}$, and a linearizing change of coordinates
$(X, Y, Z)=(x\,\ell^{-1}_{1}, y\, \ell^{-1}_{1}, z\,\ell_1\ell^{-1}_{2})$.
\bigskip

\noindent\textbf{Case $\bf 18$}:
In this case the system has an invariant algebraic plane $\ell =(1+ax+by)^2+kz(1+\frac{a}{2}x-by)=0$
with cofactor $L_{\ell}=2ax-4by+kz$. It is easy to obtain two independent first integrals
$\phi_1=x^2\,y$ and $\phi_2=y\,z^2 \,\ell^{-2}$.

\bigskip

\noindent\textbf{Case $\bf 19$}:
In this case we find two invariant algebraic surfaces $\ell_1 =(1+ax+by)^2-acxz=0$
and $\ell_2 =(1+ax+by)^2-2cz(1+ax-by-\frac{c}{2}z)=0$ with cofactors $L_{\ell_1}=2ax-4by$ and
$L_{\ell_2}=2ax-4by-2cz$. These give the first integrals
$\phi_1=x^2 \, y\, \ell^{-1}_{1}\,\ell_2$ and $\phi_2=y\,z^2 \, \ell^{-1}_{1} \ell^{-1}_{2}$.

\bigskip

\noindent\textbf{Case $\bf 20$}:
The system (\ref{1}) reduces to
\begin{equation*}
\dot x = x(1+ax+cz),\quad \dot y = y(-2+ey+fz),\quad \dot z = z(1-fz),
\end{equation*}
The first and third equations gives a linearizable node and
the change of coordinates $Y=\frac{y(1-fz)^2}{1-\frac{e}{2}y-fz}$
linearizes the second equation.

\bigskip

\noindent\textbf{Case $\bf 21$}:
The system in this case has three invariant algebraic surfaces $\ell_1 =1+ax-\frac{e}{2}y = 0$,
$\ell_2 =1+ax = 0$ and $\ell_3 =1-\frac{e}{2}y+\frac{ae}{2}xy = 0$ with cofactors $L_{\ell_1}=ax+ey$, $L_{\ell_2}=ax$ and $L_{\ell_3}=ey$.
It is easy to find two independent first integrals
$\phi_1=x^2 \, y\, \ell^{-1}_{1}$ and $\phi_2=y\,z^2 \, \ell^{\frac{a-2g}{a}}_{2} \ell^{-\frac{2h+e}{e}}_{3}$,
and a linearizing change of coordinates $(X, Y, Z)=(x\,\ell^{-1}_{2}, y\,\ell_2\ell^{-1}_{3}, z\,\ell^{-\frac{g}{a}}_{2}\ell^{-\frac{h}{e}}_{3})$.

When $a=0$, we replace $\ell_2^{1/a}$ with the exponential factor $\exp(x)$, while, when $e=0$, we replace $\ell_3^{1/e}$ by the
exponential factor $\exp(-y(1-ax)/2)$.

\qed


\end{document}